[1994/12/01]
\documentclass{ijmart-mod}
\chardef\bslash=`\\ 

\usepackage{bm}
\usepackage{graphicx}
\usepackage[breaklinks=true]{hyperref}
\usepackage{mathtools}
\usepackage{caption}
\usepackage{amsmath}
\usepackage{array}
\usepackage{multirow}
\usepackage{soul}

\newtheorem{thm}{Theorem}[section]
\newtheorem{cor}[thm]{Corollary}
\newtheorem{lem}[thm]{Lemma}
\newtheorem{prop}[thm]{Proposition}

\theoremstyle{definition}

\newtheorem{rem}[thm]{Remark}

\theoremstyle{remark}

\newcommand{\eval}[2][\right]{\relax
  \ifx#1\right\relax \left.\fi#2#1\rvert}

\begin{document}
\title{Flat simplices and kissing polytopes}

\author[A. Deza]{Antoine Deza}
\address{McMaster University, Hamilton, Ontario, Canada}
\email{deza@mcmaster.ca} 

\author[L. Pournin]{Lionel Pournin}
\address{Universit{\'e} Paris 13, Villetaneuse, France}
\email{lionel.pournin@univ-paris13.fr}

\maketitle

\begin{abstract}
We consider how flat a lattice simplex contained in the hypercube $[0,k]^d$ can be. This question is related to the notion of kissing polytopes: two lattice polytopes contained in the hypercube $[0,k]^d$ are kissing when they are disjoint but their distance is as small as possible. We show that the smallest possible distance of a lattice point $P$ contained in the cube $[0,k]^3$ to a lattice triangle in the same cube that does not contain $P$ is
$$
\frac{1}{\sqrt{3k^4-4k^3+4k^2-2k+1}}
$$
when $k$ is at least $2$. We also improve the known lower bounds on the distance of kissing polytopes for $d$ at least $4$ and $k$ at least $2$.
\end{abstract}


\section{Introduction}\label{DLP3.sec.1}

Consider a full dimensional simplex $S$ contained in $\mathbb{R}^d$. We say that two proper faces $P$ and $Q$ of $S$ are \emph{opposite} when the vertex set of $P$ and the vertex set of $Q$ form a partition of the vertex set of $S$. We are interested in measuring how close $P$ and $Q$ can be, or equivalently how flat $S$ can be, under the constraint that $S$ is a lattice $(d,k)$\nobreakdash-simplex---here and in the sequel, a lattice $(d,k)$-polytope refers to a polytope contained in $[0,k]^d$ and whose vertices are lattice points. More precisely, we will study the smallest possible distance $\varepsilon_i(d,k)$ between an $i$-dimensional face of a lattice $(d,k)$-simplex and the opposite $(d-i-1)$\nobreakdash-dimensional face. 
Similar measures of how close two polytopes can be appear in optimization \cite{BeckShtern2017,GutmanPena2018,Lacoste-JulienJaggi2015,Pena2019,PenaRodriguez2019,RademacherShu2022} and combinatorics \cite{AlonVu1997,GrahamSloane1984}. This quantity is also related to \emph{kissing polytopes}: let us denote by $\varepsilon(d,k)$ the smallest possible distance between two disjoint lattice $(d,k)$\nobreakdash-polytopes and say that two lattice $(d,k)$\nobreakdash-polytopes whose distance is $\varepsilon(d,k)$ are \emph{kissing} following the notation and terminology from \cite{DezaLiuPournin2025,DezaLiuPournin2024,DezaOnnPokuttaPournin2024}. In fact,
\begin{equation}\label{DLP3.sec.1.eq.1}
\varepsilon(d,k)=\mathrm{min}\biggl\{\varepsilon_i(d,k):0\leq{i}\leq\frac{d-1}{2}\biggr\}\mbox{.}
\end{equation}

General upper and lower bounds on $\varepsilon(d,k)$ have been proven in \cite{DezaOnnPokuttaPournin2024} and expressions for $\varepsilon(2,k)$ and $\varepsilon(3,k)$ have been established in \cite{DezaLiuPournin2025,DezaLiuPournin2024}. Moreover a computer assisted strategy allowed to compute additional values of $\varepsilon(d,k)$ when $d$ and $k$ is reasonably small \cite{DezaLiuPournin2024}. The values of $\varepsilon(d,k)$ known so far are reported in Table~\ref{DLP3.sec.1.tab.2}. It is shown in \cite{DezaLiuPournin2024} that $\varepsilon(2,k)$ is always the distance between a point and a segment. In other words, it is always equal to $\varepsilon_0(2,k)$. Note that, when $d$ is equal to $2$, there is no other possible value of $\varepsilon_i(2,k)$ because $\varepsilon_0(2,k)$ and $\varepsilon_1(2,k)$ coincide (and in general $\varepsilon_i(d,k)$ and $\varepsilon_{d-i-1}(d,k)$ are equal). When $d$ is at least $3$ however, it is shown in \cite{DezaLiuPournin2025,DezaLiuPournin2024} that all the known values of $\varepsilon(d,k)$ are equal to $\varepsilon_1(d,k)$. In these cases, $\varepsilon_0(d,k)$ is unknown, even in the cases where $\varepsilon(d,k)$ itself if known (with the exception of $\varepsilon_0(d,1)$ which is always equal to $1/\sqrt{d}$, see \cite[Lemma 2.4]{DezaOnnPokuttaPournin2024}). The first purpose of this article is to provide the value of $\varepsilon_0(d,k)$ corresponding to every known value of $\varepsilon(d,k)$ when $d$ is at least $3$ and $k$ at least $2$. In particular, we prove the following.
\begin{table}
\begin{center}
\begin{tabular}{>{\centering}p{0.8cm}cccc}
\multirow{2}{*}{$d$}& \multicolumn{4}{c}{$k$}\\
\cline{2-5}
 & $1$ & $2$  & $3$  & $k\geq4$\\
\hline & \\[-1.1\bigskipamount]
$2$ & $\sqrt{2}$ & $\sqrt{5}$ & $\sqrt{13}$  & $\sqrt{(k-1)^2+k^2}$\\
$3$ & $\sqrt{6}$ & $5\sqrt{2}$ & $\sqrt{299}$ & $\sqrt{2(2k^2-4k+5)(2k^2-2k+1)}$\\
$4$ & $3\sqrt{2}$ & $2\sqrt{113}$ & $11\sqrt{71}$\\
$5$ & $\sqrt{58}$\\
$6$ & $\sqrt{202}$\\
\end{tabular}
\end{center}
\caption{The known values of $1/\varepsilon(d,k)$.}\label{DLP3.sec.1.tab.2}
\end{table}

\begin{thm}\label{DLP3.sec.1.thm.3}
If $k$ is at least $2$, then
$$
\varepsilon_0(3,k)=\frac{1}{\sqrt{3k^4-4k^3+4k^2-2k+1}}
$$
and, up to symmetry, this distance is uniquely achieved between the point $(1,1,1)$ and the triangle with vertices $(0,0,1)$, $(k,k-1,0)$, and $(0,k,k)$
\end{thm}


When $d$ is at least $4$, the values of $\varepsilon_i(d,k)$ corresponding to the known values of $\varepsilon(d,k)$ are obtained using the computational strategy from \cite{DezaLiuPournin2024}. Tables \ref{DLP3.sec.1.tab.1} and~\ref{DLP3.sec.1.tab.1.2} show all the values of $\varepsilon_0(d,k)$ and $\varepsilon_i(d,1)$ known so far where the bolded entries correspond to the values that we report for the first time.
\begin{figure}
\begin{centering}
\includegraphics[scale=1]{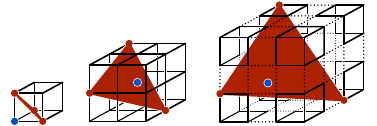}
\caption{A lattice point and a lattice triangle that achieve $\varepsilon_0(3,k)$ for $k$ equal to $1$, $2$, and at least $3$ (from left to right).}\label{DLP3.sec.1.fig.1}
\end{centering}
\end{figure}
\begin{table}[b]
\begin{center}
\begin{tabular}{>{\centering}p{0.8cm}cccc}
\multirow{2}{*}{$d$}& \multicolumn{4}{c}{$k$}\\
\cline{2-5}
 & $1$ & $2$  & $3$  & $k\geq4$\\
\hline & \\[-1.1\bigskipamount]
$2$ & $\sqrt{2}$ & $\sqrt{5}$ & $\sqrt{13}$ & $\sqrt{(k-1)^2+k^2}$\\
$3$ & $\sqrt{3}$ & $\bm{\sqrt{29}}$ & $\bm{\sqrt{166}}$ & $\bm{\sqrt{3k^4-4k^3+4k^2-2k+1}}$\\
$4$ & $2$ & $\bm{\sqrt{209}}$ & $\bm{\sqrt{3022}}$\\
$\vdots$ & $\vdots$\\
$d$ & $\sqrt{d}$\\
\end{tabular}
\end{center}
\caption{The known values of $1/\varepsilon_0(d,k)$.}\label{DLP3.sec.1.tab.1}
\end{table}

We will adopt a different point of view on how flat a lattice $(d,k)$\nobreakdash-simplex can be by considering the smallest possible distance $\varepsilon_i^u(d,k)$ between the affine hull of an $i$\nobreakdash-dimensional face of a $d$-dimensional lattice $(d,k)$\nobreakdash-simplex and the affine hull of the opposite $(d-i-1)$\nobreakdash-dimensional face, where the exponent $u$ stands for \emph{unbounded}. Ronald Graham and Neil Sloane~\cite{GrahamSloane1984} and Noga Alon and V\u{a}n V{\~u} \cite{AlonVu1997} have considered $\varepsilon_0^u(d,1)$ and shown that
$$
\frac{2^{d-1}}{\sqrt{d}^{d+3}}\leq\varepsilon_0^u(d,1)\leq\frac{4^{d+o(d)}}{\sqrt{d}^d}\mbox{.}
$$
\begin{table}
\begin{center}
\begin{tabular}{>{\centering}p{0.8cm}ccc}
\multirow{2}{*}{$d$}& \multicolumn{3}{c}{$i$}\\
\cline{2-4}
 & $0$ & $1$  & $2$\\
\hline & \\[-1.1\bigskipamount]
$2$ & $\sqrt{2}$ & &\\
$3$ & $\sqrt{3}$ & $\sqrt{6}$ &\\
$4$ & $2$ & $3\sqrt{2}$ &\\
$5$ & $\sqrt{5}$ & $\sqrt{58}$ & $\bm{\sqrt{55}}$\\
$6$ & $\sqrt{6}$ & $\sqrt{202}$ & $\bm{\sqrt{199}}$\\
\end{tabular}
\end{center}
\caption{The known values of $1/\varepsilon_i(d,1)$.}\label{DLP3.sec.1.tab.1.2}
\end{table}

In analogy with $\varepsilon_i(d,k)$, we will denote
\begin{equation}\label{DLP3.sec.1.eq.2}
\varepsilon^u(d,k)=\mathrm{min}\biggl\{\varepsilon_i^u(d,k):0\leq{i}\leq\frac{d-1}{2}\biggr\}\mbox{.}
\end{equation}

Let us first remark that all the known values of $\varepsilon(d,k)$ coincide with $\varepsilon^u(d,k)$: the strategy for the computation of all the values of $\varepsilon(d,k)$ reported in Table~\ref{DLP3.sec.1.tab.2} was to compute $\varepsilon^u(d,k)$ in order to simplify the argument by only handling affine hulls (instead of convex hulls) and then to observe that an explicit pair $P$ and $Q$ of disjoint lattice $(d,k)$-polytopes satisfies
$$
d(P,Q)=\varepsilon^u(d,k)\mbox{.}
$$
\begin{table}[b]
\begin{center}
\begin{tabular}{>{\centering}p{0.8cm}ccc}
\multirow{2}{*}{$d$}& \multicolumn{3}{c}{$k$}\\
\cline{2-4}
 & $1$ & $2$  & $3$\\
\hline & \\[-1.1\bigskipamount]
$4$ & $\bm{\sqrt{7}}$ & $\bm{\sqrt{209}}$ & $\bm{\sqrt{3142}}$\\
$5$ & $\bm{\sqrt{19}}$ & &\\
$6$ & $\bm{\sqrt{59}}$ & &\\
\end{tabular}
\end{center}
\caption{The known values of $1/\varepsilon_0^u(d,k)$.}\label{DLP3.sec.1.tab.3}
\end{table}

More precisely, it is shown, in \cite{DezaLiuPournin2024} that for all positive $k$,
$$
\varepsilon(2,k)=\varepsilon_0(2,k)=\varepsilon_0^u(2,k)
$$
and in \cite{DezaLiuPournin2025} that, for all positive $k$,
$$
\varepsilon(3,k)=\varepsilon_1(3,k)=\varepsilon_1^u(3,k)\mbox{.}
$$

We shall see here when computing $\varepsilon_0(3,k)$ that for all positive $k$,
$$
\varepsilon_0(3,k)=\varepsilon_0^u(3,k)\mbox{.}
$$

When $d$ is at least $4$, computing $\varepsilon_i^u(d,k)$ using the computer-assisted strategy from \cite{DezaLiuPournin2024} shows that $\varepsilon_i(d,k)$ and $\varepsilon_i^u(d,k)$ are not always equal. More precisely, all the known values of $\varepsilon_i^u(d,k)$ are equal to $\varepsilon_i(d,k)$ except for $\varepsilon_0^u(d,1)$ when $4\leq{d}\leq6$ and $\varepsilon_0^u(4,3)$. The known values of $\varepsilon_0^u(d,k)$ are reported in Table~\ref{DLP3.sec.1.tab.3} when $d$ is at least $4$ (the entries in Table \ref{DLP3.sec.1.tab.3} are bolded in order to indicate that all of these values are reported for the first time). Remark that $\varepsilon_0(d,1)$ is greater than $\varepsilon_0^u(d,1)$ when $d$ is at least $4$. We provide a lower bound on $\varepsilon^u(d,k)$ for arbitrary $d$ and $k$, which improves the lower bound on $\varepsilon(d,k)$ stated by Theorems 1.1 and 2.3 from \cite{DezaOnnPokuttaPournin2024}. We also recover and marginally improve the lower bound on $\varepsilon_0^u(d,1)$ by Ronald Graham and Neil Sloane~\cite{GrahamSloane1984}.

\begin{thm}\label{DLP3.sec.1.thm.1}
For every positive integer $d$,
$$
\varepsilon_0^u(d,1)\geq\frac{2^{d-1}}{\sqrt{d}^{d+1}}
$$
and for any positive integers $d$ and $k$,
$$
\varepsilon^u(d,k)\geq\frac{1}{k^{d-1}\sqrt{d}^d}\mbox{.}
$$
\end{thm}

The article is organised as follows. We extend to $\varepsilon_i(d,k)$ a number of results that have stated in \cite{DezaLiuPournin2025,DezaLiuPournin2024,DezaOnnPokuttaPournin2024} for $\varepsilon(d,k)$ and prove (\ref{DLP3.sec.1.eq.1}) in Section \ref{DLP3.sec.1.5}. We then establish Theorem \ref{DLP3.sec.1.thm.1} in Section \ref{DLP3.sec.2} and Theorem \ref{DLP3.sec.1.thm.3} in Sections \ref{DLP3.sec.3} and \ref{DLP3.sec.3.5}.

\section{Some properties of $\varepsilon_i(d,k)$ and $\varepsilon_i^u(d,k)$}\label{DLP3.sec.1.5}

By \cite[Theorem 5.1]{DezaOnnPokuttaPournin2024}, $\varepsilon(d,k)$ is a decreasing function of $d$---we mean this in the strict sense---for all fixed $k$. We first extend that property to $\varepsilon_i(d,k)$

\begin{thm}\label{DLP3.sec.1.5.thm.1}
For any positive integer $k$ and non-negative integer $i$, $\varepsilon_i(d,k)$, $\varepsilon_i^u(d,k)$, and $\varepsilon^u(d,k)$ are decreasing functions of $d$.
\end{thm}
\begin{proof}
Consider two opposite faces $P$ and $Q$ of a lattice $(d,k)$-simplex $S$ and denote by $i$ the dimension of $P$. Identify $\mathbb{R}^d$ as the hyperplane of $\mathbb{R}^{d+1}$ spanned by the first $d$ coordinates. Let $v$ be a vertex of $P$ and consider the point $\tilde{v}$ that admits $v$ as its orthogonal projection on $\mathbb{R}^d$ and whose last coordinate is equal to $1$. The convex hull of $v$ and $Q$ is a $(d-i)$-dimensional lattice $(d+1,k)$\nobreakdash-simplex disjoint from $P$, which we will denote by $\tilde{Q}$. Let us compare the distance between $P$ and $Q$ with the distance between $P$ and $\tilde{Q}$.

Let $p$ be a point in $P$ and $q$ a point in $Q$ such that the distance between $p$ and $q$ is equal to the distance between $P$ and $Q$. Pick a real number $\lambda$ that satisfies $0\leq\lambda\leq1$ and consider the two points
$$
\left\{
\begin{array}{l}
x=\lambda{v}+(1-\lambda)p\mbox{,}\\
y=\lambda\tilde{v}+(1-\lambda)q\mbox{.}
\end{array}
\right.
$$

By convexity, $x$ belongs to $P$ and $y$ belongs to $\tilde{Q}$. In particular,
\begin{equation}\label{DLP3.sec.1.5.thm.1.eq.1}
d(P,\tilde{Q})\leq{d(x,y)}\mbox{.}
\end{equation}

As the vectors $\tilde{v}-v$ and $q-p$ are orthogonal, Pythagoras's theorem yields
$$
d(x,y)^2=\lambda^2+(1-\lambda)^2d(p,q)^2
$$
and as a consequence,
$$
\frac{\partial{d(x,y)^2}}{\partial\lambda}=2\lambda\bigl(1+d(p,q)^2\bigr)-2d(p,q)^2\mbox{.}
$$

Therefore, the distance of $x$ and $y$ is a decreasing function of $\lambda$ when
\begin{equation}\label{DLP3.sec.1.5.thm.1.eq.2}
0\leq{\lambda}<\frac{d(p,q)^2}{1+d(p,q)^2}\mbox{.}
\end{equation}

As the distance of $x$ and $y$ is at least equal to the distance of $p$ and $q$ when $\lambda$ is equal to $0$, for any positive $\lambda$ satisfying (\ref{DLP3.sec.1.5.thm.1.eq.2}), the distance of $x$ and $y$ is less than the distance of $P$ and $Q$. Combining this with (\ref{DLP3.sec.1.5.thm.1.eq.1}) yields
$$
d(P,\tilde{Q})<d(P,Q)\mbox{.}
$$

Note that repeating the argument by taking for $p$ and $q$ a point in the affine hull of $P$ and a point in the affine hull of $Q$ whose distance is equal to the distance between these affine hulls shows that
$$
d\bigl(\mathrm{aff}(P),\mathrm{aff}(\tilde{Q})\bigr)<d\bigl(\mathrm{aff}(P),\mathrm{aff}(Q)\bigr)\mbox{.}
$$

By construction, $P$ and $\tilde{Q}$ are opposite faces of a $(d+1)$-dimensional lattice $(d+1,k)$-simplex. As a consequence, choosing $P$ and $Q$ in such a way that their distance is equal to $\varepsilon_i(d,k)$ shows that
$$
\varepsilon_i(d+1,k)<\varepsilon_i(d,k)
$$
and choosing them so that the distance of their affine hulls is $\varepsilon_i(d,k)$ yields
$$
\varepsilon_i^u(d+1,k)<\varepsilon_i^u(d,k)\mbox{.}
$$

It immediately follows from the latter inequality that $\varepsilon^u(d,k)$ is a decreasing function of $d$ for every fixed $i$ and $k$, as desired.
\end{proof}

The property stated by Theorem \ref{DLP3.sec.1.5.thm.1} has the following consequence.

\begin{rem}\label{DLP3.sec.1.5.rem.1}
Consider two positive integers $d$ and $k$ and an integer $i$ satisfying $0\leq{i}<d$. Recall that $\varepsilon_i(d,k)$ coincides with $\varepsilon_{d-i-1}(d,k)$. Therefore
$$
\varepsilon_i(d,k)=\varepsilon_{d-i-1}(d,k)<\varepsilon_{d-i-1}(d+1,k)=\varepsilon_{i+1}(d+1,k)
$$
\end{rem}

Let us now state a consequence of Theorem \ref{DLP3.sec.1.5.thm.1} (and Remark \ref{DLP3.sec.1.5.rem.1}) that will be useful later in order to establish an expression for $\varepsilon_0(3,k)$.

\begin{cor}\label{DLP3.sec.1.5.cor.1}
Consider two opposite faces $P$ and $Q$ of a $d$-dimensional lattice $(d,k)$-simplex such that $P$ has dimension $i$. If
$$
d(P,Q)=\varepsilon_i(d,k)
$$
then 
$$
d(P,Q)=d\bigl(\mathrm{aff}(P),\mathrm{aff}(Q)\bigr)\mbox{.}
$$
\end{cor}
\begin{proof}
Assume that the distance between $P$ and $Q$ is equal to $\varepsilon_i(d,k)$ and, for contradiction that this distance is less than the distance between the affine hulls of $P$ and $Q$. In this case, the distance between $P$ and $Q$ must be the distance between a face $F$ of $P$ and a face $G$ of $Q$ such that either $F$ is $(i-1)$-dimensional and $G$ is equal to $Q$ or $F$ is equal to $P$ and $G$ is $(d-i-2)$\nobreakdash-dimensional. In particular, there exists a hyperplane $H$ of $\mathbb{R}^d$ that contains $F$ and $G$. Denote by $a$ a non-zero vector orthogonal to $H$. Since $a$ is non-zero, the coordinates of this vector cannot be all equal to zero and we can assume without loss of generality that $a_d$ is non-zero by, if needed permuting the coordinates of $\mathbb{R}^d$, which does not change the distance between $P$ and $Q$.

Now identify $\mathbb{R}^{d-1}$ with the hyperplane spanned by the first $d-1$ coordinates of $\mathbb{R}^d$ and denote by $\pi$ the orthogonal projection from $\mathbb{R}^d$ on $\mathbb{R}^{d-1}$. Since $a_d$ is non-zero, $\pi$ induces a bijection from $H$ to $\mathbb{R}^{d-1}$. Moreover, that bijection sends a lattice $(d,k)$-polytope contained in $H$ to a lattice $(d-1,k)$-polytope. Recall that the distance between $P$ and $Q$ is equal to either the distance between $F$ and $G$. As $\pi$ is $1$-Lipschitz it follows that
\begin{equation}\label{DLP3.sec.1.5.cor.1.eq.1}
d\bigl(\pi(F),\pi(G)\bigr)\leq\varepsilon_i(d,k)\mbox{.}
\end{equation}

By construction, $\pi(F)$ and $\pi(G)$ are two opposite faces of a $(d-1)$-dimensional lattice $(d-1,k)$-simplex such that $\pi(F)$ is either $i$- or $(i-1)$-dimensional. If $\pi(F)$ is $i$-dimensional, then it follows from (\ref{DLP3.sec.1.5.cor.1.eq.1}) that $\varepsilon_i(d-1,k)$ is at most $\varepsilon_i(d,k)$ which contradicts Theorem \ref{DLP3.sec.1.5.thm.1} and if $\pi(F)$ is $(i-1)$-dimensional, then (\ref{DLP3.sec.1.5.cor.1.eq.1}) implies  that $\varepsilon_{i-1}(d-1,k)$ is at most $\varepsilon_i(d,k)$ which contradicts Remark \ref{DLP3.sec.1.5.rem.1}.
\end{proof}

\begin{rem}\label{DLP3.sec.1.5.rem.3}
For fixed $d$ and $k$, the computational procedure presented in \cite{DezaLiuPournin2024} amounts to generate all the possible pairs $P$ and $Q$ exhaustively and to compute the distance of their affine hulls. Therefore it really computes $\varepsilon_i^u(d,k)$. However, the quantity studied in \cite{DezaLiuPournin2025,DezaLiuPournin2024,DezaOnnPokuttaPournin2024} is $\varepsilon(d,k)$ and it turns out that for all the values of $d$ and $k$ investigated so far, $\varepsilon(d,k)$ coincides with $\varepsilon^u(d,k)$. Indeed, for these values of $d$ and $k$, there exists always a pair $P$ and $Q$ satifying
$$
d(P,Q)=d\bigl(\mathrm{aff}(P),\mathrm{aff}(Q)\bigr)=\varepsilon^u(d,k)
$$
which proves that $\varepsilon(d,k)$ coincides with $\varepsilon^u(d,k)$. As we mentioned in the introduction, that observation does not carry over to $\varepsilon_i(d,k)$ and $\varepsilon_i^u(d,k)$. However, it follows from Corollary \ref{DLP3.sec.1.5.cor.1} that the exhaustive enumeration procedure from \cite{DezaLiuPournin2024} allows to compute $\varepsilon_i(d,k)$ in this case nonetheless: it suffices to check, during that procedure, whether the distance of a considered pair $P$ and $Q$ coincides with the distance of their affine hulls using \cite[Remark 2]{DezaLiuPournin2024}.
\end{rem}

Now let us recall that $\varepsilon(d,k)$ is defined following \cite{DezaOnnPokuttaPournin2024} as the smallest possible distance between two disjoint lattice $(d,k)$-polytopes. However, in the introduction we have defined $\varepsilon_i(d,k)$ as the smallest possible distance between an $i$\nobreakdash-dimensional face of a $d$\nobreakdash-dimensional lattice $(d,k)$-simplex and the opposite $(d-i-1)$\nobreakdash-dimensional face. While opposite faces of a simplex are necessarily disjoint, the converse is not true and (\ref{DLP3.sec.1.eq.1}) is therefore not immediate. Theorem~5.2 from~\cite{DezaOnnPokuttaPournin2024} goes a long way towards proving (\ref{DLP3.sec.1.eq.1}). Indeed, it states that $\varepsilon(d,k)$ is achieved as the distance between two lattice $(d,k)$-simplices $P$ and $Q$ whose dimensions sum to $d-1$. There only remains to prove that $P$ and $Q$ are opposite faces of a simplex or, equivalently, that the affine hull of $P$ and $Q$ is $\mathbb{R}^d$. We show that this property holds more generally for $\varepsilon_i(d,k)$.

\begin{thm}\label{DLP3.sec.1.5.thm.2}
For any positive integers $d$ and $k$, the smallest possible distance between two disjoint lattice $(d,k)$-simplices, one of which is  $i$-dimensional and the other $(d-i-1)$-dimensional is equal to $\varepsilon_i(d,k)$.
\end{thm}
\begin{proof}
The proof is by induction on $d$. First observe that the statement is immediate when $d$ is equal to $1$ because $\varepsilon_0(1,k)$ is equal to~$1$, the smallest possible distance between two distinct integers. Assume that $d$ is at least $2$ and consider two disjoint lattice $(d,k)$-simplices $P$ and $Q$ such that $P$ has dimension $i$ and $Q$ dimension $d-i-1$. Assume for contradiction that
\begin{equation}\label{DLP3.sec.1.5.thm.2.eq.-3}
d(P,Q)<\varepsilon_i(d,k)\mbox{.}
\end{equation}

It follows, by the definition of $\varepsilon_i(d,k)$, that $P$ and $Q$ cannot be opposite faces of a $d$-dimensional simplex. Equivalently, there exists a hyperplane $H$ of $\mathbb{R}^d$ that contains $P$ and $Q$. Let $a$ be a non-zero vector orthogonal to $H$. Since $a$ is non-zero, so is one of the coordinates of $a$, and we can assume without loss of generality that $a_d$ is non-zero by permuting the coordinates of $\mathbb{R}^d$ if needed (which does not change the distance between $P$ and $Q$). Now, identify $\mathbb{R}^{d-1}$ with the hyperplane of $\mathbb{R}^d$ spanned by the first $d-1$ coordinates and denote by $\pi$ the orthogonal projection from $\mathbb{R}^d$ on $\mathbb{R}^{d-1}$. As $a_d$ is non-zero, $\pi$ induces a bijection between $H$ and $\mathbb{R}^{d-1}$ and since $\pi$ is $1$-Lipschitz,
\begin{equation}\label{DLP3.sec.1.5.thm.2.eq.-2}
d\bigl(\pi(P),\pi(Q)\bigr)\leq{d(P,Q)}\mbox{.}
\end{equation}

Consider a point $p$ in $P$ and a point $q$ in $Q$ such that the distance between $\pi(P)$ and $\pi(Q)$ is equal to the distance between $\pi(p)$ and $\pi(q)$. Observe that if $p$ is in a proper face of $P$, then by induction
\begin{equation}\label{DLP3.sec.1.5.thm.2.eq.-1}
\varepsilon_{i-1}(d-1,k)\leq{d\bigl(\pi(P),\pi(Q)\bigr)}
\end{equation}
and similarly, if $q$ is in a proper face of $Q$, then by induction
\begin{equation}\label{DLP3.sec.1.5.thm.2.eq.0}
\varepsilon_i(d-1,k)\leq{d\bigl(\pi(P),\pi(Q)\bigr)}\mbox{.}
\end{equation}

In the former case, combining (\ref{DLP3.sec.1.5.thm.2.eq.-3}), (\ref{DLP3.sec.1.5.thm.2.eq.-2}), and (\ref{DLP3.sec.1.5.thm.2.eq.-1}) contradicts Remark \ref{DLP3.sec.1.5.rem.1}. In the latter case, combining (\ref{DLP3.sec.1.5.thm.2.eq.-3}), (\ref{DLP3.sec.1.5.thm.2.eq.-2}), and (\ref{DLP3.sec.1.5.thm.2.eq.0}) contradicts Theorem \ref{DLP3.sec.1.5.thm.1}. It therefore suffices to prove that the distance between $\pi(P)$ and $\pi(Q)$ is necessaily achieved by a point in $\pi(P)$ and a point in $\pi(Q)$ one of which is in a proper face of the corresponding simplex. Let us assume that $p$ and $q$ are each contained in the relative interior of the corresponding simplex (otherwise we are done). In that case, observe that $\pi(P)$ and $\pi(Q)$ are both necessarily orthogonal to $\pi(q-p)$. As a consequence $\pi(P)-\pi(p)$ and $\pi(Q)-\pi(q)$ are two simplices contained in the linear hyperplane of $\mathbb{R}^{d-1}$ orthogonal to $\pi(q-p)$. As the dimensions of these two simplices sum to $d-1$, their intersection must be a polytope of positive dimension. Pick a point $x$ in a proper face of that intersection. By construction, either $x+\pi(p)$ is in a proper face of $\pi(P)$ or $x+\pi(q)$ is in a proper face of $\pi(Q)$. As the distance between $x+\pi(p)$ and $x+\pi(q)$ is equal to the distance between $\pi(p)$ and $\pi(q)$, this completes the proof.
\end{proof}

The remainder of the section is devoted to recalling the algebraic model for $\varepsilon(d,k)$ introduced in \cite{DezaLiuPournin2024}. This model, that we will state in the case of $\varepsilon_i^u(d,k)$, will be one of the main ingredients in the proofs of Sections~\ref{DLP3.sec.2},~\ref{DLP3.sec.3}, and~\ref{DLP3.sec.3.5}. Consider two opposite faces $P$ and $Q$ of a lattice $(d,k)$-simplex $S$ and denote by $i$ the dimension of $P$. Further denote by $p^0$ to $p^i$ the vertices of $P$ and by $q^0$ to $q^{d-i-1}$ the vertices of $Q$ and consider the $d\mathord{\times}(d-1)$ matrix $A$ whose column $j$ is the vector $p^j-p^0$ if $j$ is at most $i$ and the vector $q^{j-i}-q^0$ otherwise. Since $P$ and $Q$ are opposite faces of $S$ these vectors collectively span a hyperplane of $\mathbb{R}^d$ and therefore $A$ has rank $d-1$. Further denote by $b$ the vector $q^0-p^0$. According to \cite[Lemma 2]{DezaLiuPournin2024}, if $A^tA$ is non-singular, then
$$
d\bigl(\mathrm{aff}(P),\mathrm{aff}(Q)\bigr)=\bigl\|A(A^tA)^{-1}A^tb-b\bigr\|\mbox{.}
$$

In our case, though, $A^tA$ is necessarily non-singular. Indeed, by the Cauchy--Binet formula  \cite[Example 10.31]{ShafarevichRemizov2013} the determinant of this matrix is
\begin{equation}\label{DLP3.sec.1.5.eq.1}
\mathrm{det}\bigl(A^tA\bigr)=\sum_{j=1}^d\mathrm{det}(A_j)^2
\end{equation}
where $A_j$ denotes the $(d-1)\mathord{\times}(d-1)$ matrix obtained by removing row $j$ from the matrix $A$. Since $A$ has rank $d-1$, one of the matrices $A_1$ to $A_d$ is non-singular and (\ref{DLP3.sec.1.5.eq.1}) implies that $A^tA$ is non-singular as well. 
\begin{rem}
With these notation, the lattice vector
$$
a=\bigl(\mathrm{det}(A_1),-\mathrm{det}(A_2),\ldots,(-1)^{d+1}\mathrm{det}(A_d)\bigr)
$$
is  orthogonal to both $P$ and $Q$. Indeed, $a\mathord{\cdot}(p^j-p^0)$ is the determinant of the $d\mathord{\times}d$ matrix obtained by prepending $p^j-p^0$ to $A$ as a first column. That matrix is necessarily singular since two of its columns coincide, proving that $a\mathord{\cdot}(p^j-p^0)$ is equal to $0$ and therefore that $a$ is orthogonal to $P$. The same argument shows that $a\mathord{\cdot}(q^j-q^0)$ also vanishes and that $a$ is orthogonal to $Q$ as well.
\end{rem}

The entries of $A$ each are a difference between two non-negative numbers at most $k$ and therefore belong to the interval $[-k,k]$. If $k$ is equal to $1$, one can always assume, thanks to the symmetries of the unit hypercube $[0,1]^d$ that $q^0$ is the origin of $\mathbb{R}^d$. If in addition, $i$ is equal to $0$, then the columns of $A$ are the vectors $q^1$ to $q^{d-1}$ and therefore, each entry of $A$ is either $0$ or $1$. We obtain the following statement as a consequence of these observations.

\begin{lem}\label{DLP3.sec.1.5.lem.1}
There exist a $d\mathord{\times}(d-1)$ matrix $A$ and a vector $b$ from $\mathbb{R}^d$ both with integer entries of absolute value at most $k$ such that
\begin{enumerate}
\item[(i)] $A^tA$ is non-singular and
\item[(ii)] $\varepsilon_i^u(d,k)=\bigl\|A(A^tA)^{-1}A^tb-b\bigr\|$.
\end{enumerate}

If in addition, $i$ is equal to $0$ and $k$ is equal to $1$, then it can be further required that the matrix $A$ has binary entries.
\end{lem}

\begin{rem}\label{DLP3.sec.1.5.rem.2}
Even though we will not make use of this in the sequel, Lemma~\ref{DLP3.sec.1.5.lem.1} still holds when $\varepsilon_i^u(d,1)$ is replaced with $\varepsilon_i(d,1)$, due to Corollary \ref{DLP3.sec.1.5.cor.1}. In turn, this lemma further holds when $\varepsilon_i^u(d,1)$ is replaced by $\varepsilon(d,1)$ according to Theorem \ref{DLP3.sec.1.5.thm.2}, and when $\varepsilon_i^u(d,1)$ is replaced by $\varepsilon^u(d,1)$ by (\ref{DLP3.sec.1.eq.2}).
\end{rem}

\section{An improved lower bound on $\varepsilon(d,k)$}\label{DLP3.sec.2}

This section is devoted to improving the lower bound on $\varepsilon(d,k)$ from \cite{DezaOnnPokuttaPournin2024} and the lower bound on $\varepsilon_0^u(d,1)$ from~\cite{GrahamSloane1984}. The former improvement will in fact be obtained by lower bounding $\varepsilon_i^u(d,k)$ independently from $i$, therefore providing a lower bound on $\varepsilon^u(d,k)$ and in turn on $\varepsilon(d,k)$. All the proofs are based on a refinement of Lemma~\ref{DLP3.sec.1.5.lem.1} obtained from arguments on matrix algebra.


\begin{prop}\label{DLP3.sec.2.prop.1}
If $M$ is a symmetric $n\mathord{\times}n$ idempotent matrix with real coefficients and $b$ is a vector from $\mathbb{R}^n$, then
$$
\|Mb-b\|^2=\|b\|^2-b^tMb\mbox{.}
$$
\end{prop}
\begin{proof}
Denote by $I$ the $n\mathord{\times}n$ identity matrix. If $M$ is symmetric, then
$$
\begin{array}{rcl}
\displaystyle\|Mb-b\|^2 &\!\!\! = \!\!\!\!& \displaystyle(Mb-b)^t(Mb-b)\\
 &\!\!\! = \!\!\!\!& \displaystyle b^t(M-I)^t(M-I)b\\
 &\!\!\! = \!\!\!\!& \displaystyle b^t(M-I)^2b\mbox{.}
\end{array}
$$

If in addition $M$ is idempotent, then
$$
\begin{array}{rcl}
\displaystyle(M-I)^2 &\!\!\! = \!\!\!\!& \displaystyle M^2-2M+I\\
 &\!\!\! = \!\!\!\!& \displaystyle I-M
\end{array}
$$
and the proposition follows.
\end{proof}

Recall that Lemma \ref{DLP3.sec.1.5.lem.1} provides a formula for $\varepsilon_i^u(d,k)$ as a function of $A$ and $b$. Thanks to Proposition \ref{DLP3.sec.2.prop.1} we can refine this formula as follows.

\begin{lem}\label{DLP3.sec.2.lem.1}
There exist a $d\mathord{\times}(d-1)$ matrix $A$ and a vector $b$ from $\mathbb{R}^d$ both with integer entries of absolute value at most $k$ such that
\begin{enumerate}
\item[(i)] $A^tA$ is non-singular and
\item[(ii)] $\varepsilon_i^u(d,k)=\sqrt{\|b\|^2-b^tA(A^tA)^{-1}A^tb}$.
\end{enumerate}

If in addition, $i$ is equal to $0$ and $k$ is equal to $1$, then it can be further required that the matrix $A$ has binary entries.
\end{lem}
\begin{proof}
According to Lemma \ref{DLP3.sec.1.5.lem.1}, there exists a $d\mathord{\times}(d-1)$ matrix $A$ and a vector $b$ from $\mathbb{R}^d$ such that $A^tA$ is non-singular and
\begin{equation}\label{DLP3.sec.2.lem.1.eq.1}
\varepsilon_i^u(d,k)=\|A(A^tA)^{-1}A^tb-b\|\mbox{.}
\end{equation}

Moreover, both $A$ and $b$ have integer entries of absolute value at most $k$. Since $A(A^tA)^{-1}A^t$ is symmetric and idempotent, Proposition \ref{DLP3.sec.2.prop.1} yields
$$
\|A(A^tA)^{-1}A^tb-b\|^2=\|b\|^2-b^tA(A^tA)^{-1}A^tb\mbox{.}
$$

Combining this with (\ref{DLP3.sec.2.lem.1.eq.1}) completes the proof.
\end{proof}

\begin{rem}\label{DLP3.sec.2.rem.1}
According to Remark \ref{DLP3.sec.1.5.rem.2}, Lemma \ref{DLP3.sec.2.lem.1} still holds when, in its statement, $\varepsilon_i^u(d,k)$ is replaced by $\varepsilon_i(d,k)$, $\varepsilon^u(d,k)$, or $\varepsilon(d,k)$.
\end{rem}

We shall now lower bound $\varepsilon_i^u(d,k)$ in terms of the determinant of $A^tA$. Let us point out that by (\ref{DLP3.sec.1.5.eq.1}), this determinant is non-negative.

\begin{thm}\label{DLP3.sec.2.thm.1}
There exists a $d\mathord{\times}(d-1)$ matrix $A$ with integer entries of absolute value at most $k$ such that $A^tA$ is non-singular and
$$
\varepsilon^u_i(d,k)\geq\frac{1}{\sqrt{\det(A^tA)}}\mbox{.}
$$

If in addition, $i$ is equal to $0$ and $k$ is equal to $1$, then it can be further required that the matrix $A$ has binary entries.
\end{thm}
\begin{proof}
According to Lemma \ref{DLP3.sec.2.lem.1}, there exists a $d\mathord{\times}(d-1)$ matrix $A$ and a vector $b$ from $\mathbb{R}^d$, both of which have integer entries of absolute value at most $k$, such that $A^tA$ is non-singular and $\varepsilon_i^u(d,k)$ can be expressed as
\begin{equation}\label{DLP3.sec.2.thm.1.eq.1}
\varepsilon_i^u(d,k)=\sqrt{\|b\|^2-b^tA(A^tA)^{-1}A^tb}\mbox{.}
\end{equation}

Moreover, if $i$ and $k$ are equal to $0$ and $1$, respectively, then $A$ has binary entries. Denoting by $C$ the cofactor matrix of $A^tA$,
$$
(A^tA)^{-1}=\frac{C^t}{\det(A^tA)}\mbox{.}
$$

As a consequence,
$$
\|b\|^2-b^tA(A^tA)^{-1}A^tb=\frac{\|b\|^2\det(A^tA)-b^tAC^tA^tb}{\det(A^tA)}\mbox{.}
$$

However, since the entries of $A$, $b$, and $C$ are integers, the numerator in the right-hand side is an integer. As in addition the left-hand side is positive and, according to (\ref{DLP3.sec.1.5.eq.1}), the determinant of $A^tA$ is non-negative, 
$$
\|b\|^2-b^tA(A^tA)^{-1}A^tb\geq\frac{1}{\det(A^tA)}\mbox{.}
$$

Combining this with (\ref{DLP3.sec.2.thm.1.eq.1}) completes the proof
\end{proof}

We conclude the section by proving Theorem \ref{DLP3.sec.1.thm.1}.

\begin{proof}[Proof of Theorem \ref{DLP3.sec.1.thm.1}]
According to Theorem \ref{DLP3.sec.2.thm.1},
$$
\varepsilon_i^u(d,k)\geq\frac{1}{\sqrt{\det(A^tA)}}\mbox{,}
$$
where $A$ is a $d\mathord{\times}(d-1)$ matrix with integer entries of absolute value at most $k$ such that $A^tA$ is non-singular. Moreover, if $i$ is equal to $0$ and $k$ to $1$, then $A$ has binary entries. By (\ref{DLP3.sec.1.5.eq.1}), this inequality can rewritten as
\begin{equation}\label{DLP3.sec.1.thm.1.eq.3}
\varepsilon_i^u(d,k)\geq\frac{1}{\sqrt{\displaystyle\sum_{j=1}^d\det(A_j)^2}}\mbox{.}
\end{equation}
where $A_j$ is the $(d-1)\mathord{\times}(d-1)$ matrix obtained from $A$ by removing the $j$th row. However, by Hadamard's inequality,
$$
|\det(A_j)|\leq{k}^{d-1}\sqrt{d-1}^{d-1}
$$
and it follows from (\ref{DLP3.sec.1.thm.1.eq.3}) that
$$
\varepsilon_i^u(d,k)\geq\frac{1}{\sqrt{d}k^{d-1}\sqrt{d-1}^{d-1}}
$$
which implies the second announced inequality. Finally, assume that $i$ is equal to $0$ and that $k$ is equal to $1$. In this case, $A$ and therefore $A_j$ have binary entries. Hadamard's inequality can then be improved into
$$
|\det(A_j)|\leq\frac{\sqrt{d}^d}{2^{d-1}}
$$
which, combined with (\ref{DLP3.sec.1.thm.1.eq.3}), proves the first announced inequality.
\end{proof}

\section{The distance of a lattice point to a lattice triangle}\label{DLP3.sec.3}

It is shown in \cite{DezaLiuPournin2025} that when $k$ is at least $4$, then
\begin{equation}\label{DLP3.sec.3.eq.1}
\varepsilon_1(3,k)=\frac{1}{\sqrt{2(2k^2-4k+5)(2k^2-2k+1)}}\mbox{.}
\end{equation}

It is also shown that $\varepsilon(3,k)$ is always strictly less than $\varepsilon_0(3,k)$ for all positive $k$ (see \cite[Theorem 1.2]{DezaLiuPournin2025}). It follows in particular that
$$
\varepsilon(3,k)=\varepsilon_1(3,k)<\varepsilon_0(3,k)\mbox{.}
$$

The purpose of this section is to establish an expression similar to (\ref{DLP3.sec.3.eq.1}) but for $\varepsilon_0(3,k)$. Let us first remark that the point $P^\star$ of coordinates $(1,1,1)$ and the triangle $Q^\star$ whose vertices are $(0,0,1)$, $(k,k-1,0)$, and $(0,k,k)$ satisfy
\begin{equation}\label{DLP3.sec.3.eq.2}
d(P^\star,Q^\star)=\frac{1}{\sqrt{3k^4-4k^3+4k^2-2k+1}}
\end{equation}
for every $k$ at least $2$. In particular, $\varepsilon_0(3,k)$ is at most the right-hand side of~(\ref{DLP3.sec.3.eq.2}). We shall show that this upper bound on $\varepsilon_0(3,k)$ is always sharp.

When $k$ is reasonably small, $\varepsilon_0(3,k)$ can be obtained computationally by the exhaustive enumeration procedure from \cite{DezaLiuPournin2024}. This procedure also allows to recover all the lattice points in $[0,k]^3$ and lattice $(3,k)$-triangles, whose distance is equal to $\varepsilon_0(3,k)$. We obtain the following from that procedure.

\begin{prop}\label{DLP3.sec.3,prop.-1}
If $k$ is at least $2$ and at most $8$, then
$$
\varepsilon_0(3,k)=\frac{1}{\sqrt{3k^4-4k^3+4k^2-2k+1}}
$$
and, up to symmetry, this distance is uniquely achieved by $P^\star$ and $Q^\star$.
\end{prop}

Note that the exhaustive enumeration procedure that provides Proposition~\ref{DLP3.sec.3,prop.-1} also shows that  $\varepsilon_0(3,k)$ coincides with $\varepsilon_0^u(3,k)$ when $2\leq{k}\leq8$.

We will extend the statement of Proposition \ref{DLP3.sec.3,prop.-1} to every integer $k$ greater than $8$. Consider a lattice point $P$ contained in the cube $[0,k]^3$ and a lattice $(3,k)$-triangle $Q$ such that $P$ is not contained in the affine hull of $Q$ (so that $P$ and $Q$ are indeed opposite faces of a lattice tetrahedron). Denote by $q^0$, $q^1$, and $q^2$ the vertices of $Q$. Following the algebraic model from \cite{DezaLiuPournin2024} that we recalled at the end of Section \ref{DLP3.sec.1.5}, we will denote by $A$ the $3\mathord{\times}2$ matrix whose first column is $q^1-q^0$ and whose second column is $q^2-q^0$ and by $b$ the vector $q^0-P$. As explained in Section \ref{DLP3.sec.1.5}, $A^tA$ is non-singular and
\begin{equation}\label{DLP3.sec.3.eq.3}
d\bigl(\mathrm{aff}(P),\mathrm{aff}(Q)\bigr)=\bigl\|A(A^tA)^{-1}A^tb-b\bigr\|\mbox{.}
\end{equation}


Now consider the lattice point $x$ from the hypercube $[-k,k]^9$ whose coordinates can be obtained from the entries of $A$ and $b$ by identification as
\begin{equation}\label{DLP3.sec.3.eq.4}
A=\left[
\begin{array}{cc}
x_1 & x_4\\
x_2 & x_5\\
x_3 & x_6\\
\end{array}
\right]
\end{equation}
and
\begin{equation}\label{DLP3.sec.3.eq.5}
b=\left[
\begin{array}{cc}
x_7\\
x_8\\
x_9\\
\end{array}
\right]\!\!\mbox{.}
\end{equation}

The determinant of $A^tA$ is then the function $g$ of $x$ defined by 
\begin{equation}\label{DLP3.sec.3.eq.7}
g(x)=(x_1x_5-x_2x_4)^2+(x_1x_6-x_3x_4)^2+(x_2x_6-x_3x_5)^2\mbox{.}
\end{equation}

This determinant is non-zero and a sum of squares and therefore $g(x)$ is always positive. With this notation, the right-hand side of (\ref{DLP3.sec.3.eq.3}) can be expressed as
\begin{equation}\label{DLP3.sec.3.eq.6}
\|A(A^tA)^{-1}A^tb-b\|=\frac{|f(x)|}{\sqrt{g(x)}}
\end{equation}
where $f$ is the function of $x$ defined by
\begin{equation}\label{DLP3.sec.3.eq.8}
f(x)=(x_1x_5-x_2x_4)x_9-(x_1x_6-x_3x_4)x_8+(x_2x_6-x_3x_5)x_7\mbox{.}
\end{equation}

It should be noted that up to six different points $x$ may be obtained from $P$ and $Q$ by permuting $q^0$, $q^1$, and $q^2$. We shall denote by $\mathcal{X}_{P,Q}(k)$ all the lattice points $x$ in $[-k,k]^9$ that can be obtained from $P$ and $Q$ as we have just described by possibly permuting $q^0$, $q^1$, and $q^2$ or by applying to both $P$ and $Q$ an isometry of $[0,k]^3$. Note that these isometries send lattice $(3,k)$-polytopes to lattice $(3,k)$-polytopes as they consist in permuting the coordinates of $\mathbb{R}^3$ and perform symmetries with respect to the planes of the form
$$
\{x\in\mathbb{R}^3:x_i=k/2\}
$$
where $i$ is equal to either $1$, $2$, or $3$. Our strategy consists in showing that, when $k$ is at least $9$ and the distance between the affine hulls of $P$ and $Q$ is $\varepsilon_0^u(3,k)$, some point $x$ in $\mathcal{X}_{P,Q}(k)$ can be recovered from a finite set $\mathcal{A}$ of lattice points independent from $k$. More precisely, $\mathcal{A}$ will be the set of the lattice points $a$ contained in $[0,6]^2\mathord{\times}[-4,6]\mathord{\times}[0,6]^3$ such that $a_6$ is at least $-a_3$.

In order to relate $\mathcal{A}$ with $\mathcal{X}_{P,Q}(k)$ we define the affine map $\phi_k:\mathbb{R}^6\rightarrow\mathbb{R}^6$ such that the $i$th coordinate of $\phi_k(a)$ is given by
$$
\bigl[\phi_k(a)\bigr]_i=\left\{
\begin{array}{l}
a_i\mbox{ if }i\mbox{ is equal to }3\mbox{ or }4\mbox{,}\\
k-a_i\mbox{ otherwise.}\\
\end{array}
\right.
$$

From now on, we identify $\mathbb{R}^6$ with the subspace of $\mathbb{R}^9$ spanned by the first six coordinates. Section~\ref{DLP3.sec.3.5} is dedicated to establish the following theorem that provides the announced reduction from $\mathcal{X}_{P,Q}(k)$ to $\mathcal{A}$.

\begin{thm}\label{DLP3.sec.3.thm.1}
Consider a lattice point $P$ contained in the cube $[0,k]^3$ and a lattice $(3,k)$-triangle $Q$. If $k$ is at least $9$ and the distance between $P$ and the affine hull of $Q$ is equal to $\varepsilon_0^u(3,k)$, then there exists a point in $\mathcal{X}_{P,Q}(k)$ whose othogonal projection on $\mathbb{R}^6$ is contained in $\phi_k(\mathcal{A})$.
\end{thm}

In the remainder of the section, we explain how Theorem \ref{DLP3.sec.3.thm.1} allows to compute $\varepsilon_0(3,k)$ when $k$ is at least $9$. Let us first state two propositions. The first one is established in the proof of \cite[Lemma 3.2]{DezaLiuPournin2025}.

\begin{prop}\label{DLP3.sec.3.prop.-0.5}
Consider a lattice point $P$ contained in the cube $[0,k]^3$ and a lattice $(3,k)$-triangle $Q$. If the affine hull of $Q$ does not contain $P$, then for every point $x$ in $\mathcal{X}_{P,Q}(k)$, the three squares $(x_1x_5-x_2x_4)^2$, $(x_1x_6-x_3x_4)^2$, and $(x_2x_6-x_3x_5)^2$ are each less than or equal to $k^4$.
\end{prop}

The second proposition follows from (\ref{DLP3.sec.3.eq.3}), (\ref{DLP3.sec.3.eq.6}), and Proposition~\ref{DLP3.sec.3.prop.-0.5}.

\begin{prop}\label{DLP3.sec.3.prop.0}
Consider a lattice point $P$ contained in the cube $[0,k]^3$ and a lattice $(3,k)$-triangle $Q$. If the distance between $P$ and the affine hull of $Q$ is equal to $\varepsilon_0^u(3,k)$, then for every point $x$ in $\mathcal{X}_{P,Q}(k)$, 
\begin{enumerate}
\item[(i)] $f(x)$ is equal to $1$ or to $-1$ and
\item[(ii)] $g(x)$ is equal to $1/\varepsilon_0^u(3,k)^2$.
\end{enumerate}
\end{prop}

\begin{proof}
Assume that the distance between $P$ and the affine hull of $Q$ is equal to $\varepsilon_0^u(3,k)$. It follows from (\ref{DLP3.sec.3.eq.3}) and (\ref{DLP3.sec.3.eq.6}) that
\begin{equation}\label{DLP3.sec.3.lem.1.eq.-3}
\varepsilon_0^u(3,k)=\frac{|f(x)|}{\sqrt{g(x)}}\mbox{.}
\end{equation}

Since $f(x)$ is an integer and $\varepsilon_0^u(3,k)$ is non-zero, the absolute value of $f(x)$ is therefore at least $1$. There remains to show that this absolute value is at most~$1$. Not only will this show that $f(x)$ is equal to $1$ or $-1$ but also that $g(x)$ is the squared inverse of $\varepsilon_0^u(3,k)$ because of (\ref{DLP3.sec.3.lem.1.eq.-3}).

Assume for contradiction that the absolute value of $f(x)$ is greater than~$1$. In that case, (\ref{DLP3.sec.3.lem.1.eq.-3}) and the integrality of $f(x)$ imply that
$$
\varepsilon_0^u(3,k)\geq\frac{2}{\sqrt{g(x)}}\mbox{.}
$$

However, it follows from (\ref{DLP3.sec.3.eq.7}) and Proposition~\ref{DLP3.sec.3.prop.-0.5} that $g(x)$ is at most $3k^4$ and we get that $\varepsilon_0^u(3,k)$, and therefore $\varepsilon_0(d,k)$, is at least $2/(\sqrt{3}k^2)$ which contradicts the fact that $\varepsilon_0(3,k)$ is at most the right-hand side of (\ref{DLP3.sec.3.eq.2}).
\end{proof}

Let $\tilde{f}:\mathbb{Z}^9\mathord{\setminus}\{0\}\rightarrow\mathbb{N}$ be the map such that $\tilde{f}(x)$ is the greatest common divisor of $x_1x_5-x_2x_4$, $x_1x_6-x_3x_4$, and $x_2x_6-x_3x_5$. 
%
According to (\ref{DLP3.sec.3.eq.8}), if $|f(x)|$ is equal to $1$, these three quantities must be relatively prime and, therefore, $\tilde{f}(x)$ is also necessarily equal to $1$. In other words, Proposition~\ref{DLP3.sec.3.prop.0} still holds when $f$ is replaced by $\tilde{f}$ in its statement. Now, a crucial observation is that $\tilde{f}(x)$ and $g(x)$ only depend on the first six coordinates of $x$. In particular, combining Theorem~\ref{DLP3.sec.3.thm.1} and Proposition~\ref{DLP3.sec.3.prop.0} allows to compute $\varepsilon_0(3,k)$ by studying the finitely-many functions $k\rightarrow\tilde{f}\circ\phi_k(a)$ and $k\rightarrow{g\circ\phi_k(a)}$ when $a$ ranges over $\mathcal{A}$. Let us explain how that can be done using symbolic computation. First observe that $\tilde{f}\circ\phi_k(a)$ is the absolute value of the greatest common divisor of 
$$
\left\{
\begin{array}{l}
k^2-(a_1+a_4+a_5)k+a_1a_5+a_2a_4\mbox{,}\\
k^2-(a_1+a_6)k+a_1a_6-a_3a_4\mbox{,}\\
k^2-(a_2+a_3+a_6)k+a_2a_6+a_3a_5\mbox{.}\\
\end{array}
\right.
$$

For any given point $a$ in $\mathcal{A}$, one can check using symbolic computation that, when the greatest common divisor of these three quadratic polynomials is a linear or quadratic function of $k$, its absolute value is never equal to $1$ when $k$ is at least $9$. This just amounts to solve two linear or quadratic equations which is easily done using symbolic computation. Hence, it follows from Theorem~\ref{DLP3.sec.3.thm.1} and Proposition \ref{DLP3.sec.3.prop.0} that $\varepsilon_0^u(3,k)$ is the inverse of $\sqrt{g\circ\phi_k(a)}$ for some point $a$ in $\mathcal{A}$ such that  are such that the greatest common divisor of the above three quadratic polynomials has degree $0$. Now consider the set of all the different polynomials $k\rightarrow{g\circ\phi_k(a)}$ obtained when $a$ ranges over that particular subset of $\mathcal{A}$. Symbolic computation shows that the largest real root of the difference between any two such polynomials is less than $9$. One of these polynomials is therefore greater than all the others when $k$ ranges over $[9,+\infty[$. This polynomial turns out to be $3k^4-4k^3+4k^2-2k+1$ and it is obtained for just the four points $(0,0,-1,1,0,1)$, $(0,0,1,0,1,0)$, $(0,1,-1,0,0,1)$, and $(0,1,0,1,0,0)$ from $\mathcal{A}$. 
Using these computational results, we now prove Theorem \ref{DLP3.sec.1.thm.3}.

\begin{proof}[Proof of Theorem \ref{DLP3.sec.1.thm.3}]
When $k$ is at least $2$ and at most $8$, the theorem follows from Proposition \ref{DLP3.sec.3,prop.-1} and we will therefore assume from now on that $k$ is at least $9$. First observe that according to (\ref{DLP3.sec.3.eq.2}),
\begin{equation}\label{DLP3.sec.1.thm.1.eq.0}
\varepsilon_0(3,k)\leq\frac{1}{\sqrt{3k^4-4k^3+4k^2-2k+1}}\mbox{.}
\end{equation}

We shall prove the reverse inequality. Consider a lattice point $P$ contained in the cube $[0,k]^3$ and a lattice $(3,k)$-triangle $Q$ such that the distance between $P$ and the affine hull of $Q$ is equal to $\varepsilon_0^u(3,k)$. According to Theorem~\ref{DLP3.sec.3.thm.1}, there exists a point $x$ in $\mathcal{X}_{P,Q}(k)$ and a point $a$ in $\mathcal{A}$ such that $\phi_k(a)$ is the orthogonal projection of $x$ on $\mathbb{R}^6$. By Proposition~\ref{DLP3.sec.3.prop.0}, the absolute value of $f(x)$, and therefore $\tilde{f}\circ\phi_k(a)$, is equal to $1$. Moreover, $g(x)$, and therefore $g\circ\phi_k(a)$, is the inverse of $\varepsilon_0^u(3,k)^2$. As $\varepsilon_0^u(3,k)$ is at most $\varepsilon_0(3,k)$, (\ref{DLP3.sec.1.thm.1.eq.0}) implies
\begin{equation}\label{DLP3.sec.1.thm.1.eq.1}
g\circ\phi_k(a)\geq3k^4-4k^3+4k^2-2k+1\mbox{.}
\end{equation}

However, as $\tilde{f}\circ\phi_k(a)$ is equal to $1$ the above computational results imply that (\ref{DLP3.sec.1.thm.1.eq.1}) is an equality and, as a consequence,
$$
\varepsilon_0(3,k)\geq\varepsilon_0^u(3,k)=\frac{1}{\sqrt{3k^4-4k^3+4k^2-2k+1}}
$$
which combined with (\ref{DLP3.sec.1.thm.1.eq.0}) provides the desired expression for $\varepsilon_0(3,k)$ and shows that $\varepsilon_0^u(3,k)$ coincides with $\varepsilon_0(3,k)$. In addition, $a$ is one of the four points $(0,0,-1,1,0,1)$, $(0,0,1,0,1,0)$, $(0,1,-1,0,0,1)$, and $(0,1,0,1,0,0)$. The image of each of these points by $\phi_k$ provides the first six coordinates of a lattice point in the hypercube $[-k,k]^9$ from which a lattice triangle $Q$ can be reconstructed. All the triangles that can be reconstructed from these points coincide, up to the symmetries of $[0,k]^3$ to the triangle $Q^\star$, which proves that $Q$ coincides with $Q^\star$, up to symmetry. There remains to show that $P^\star$ is the only lattice point in the cube $[0,k]^3$ whose distance with $Q^\star$ is equal to $\varepsilon_0(3,k)$.

Assume for contradiction that $P$ is not equal to $P^\star$ but has the same distance to $Q^\star$ than $P^\star$ and observe that by Corollary \ref{DLP3.sec.1.5.cor.1},
$$
d(P,Q^\star)=d\bigl(\mathrm{aff}(P),\mathrm{aff}(Q^\star)\bigr)\mbox{.}
$$

As a consequence, the projection of both $P$ and $P^\star$ on the affine hull of $Q^\star$ must belong to $Q^\star$. In turn, this means that one of the points obtained by adding or subtracting $P^\star-P$ to each vertex of $Q^\star$ is a lattice point $q$ contained in $Q^\star$. Observe that $q$ cannot be a vertex of $Q^\star$. Indeed otherwise, the line segment with extremities $P$ and $P^\star$ would be a translate of an edge of $Q^\star$. In that case, the orthogonal projections of $P$ and $P^\star$ on the affine hull of $Q^\star$ are necessarily two vertices of $Q^\star$ because otherwise, one of them would be outside of $Q^\star$. This would imply that $\varepsilon_0(3,k)$ is the distance between two lattice points and therefore at least $1$. Hence, $q$ is a lattice point contained in $Q^\star$ and distinct from its three vertices. But in that case, $Q^\star$ contains a strictly smaller lattice $(3,k)$\nobreakdash-triangle whose distance to $P$ or $P^\star$ is $\varepsilon_0(3,k)$, which is impossible because any such triangle must coincide with $Q^\star$ up to symmetry.
\end{proof}

\section{The proof of Theorem \ref{DLP3.sec.3.thm.1}}\label{DLP3.sec.3.5}

In order to prove Theorem \ref{DLP3.sec.3.thm.1}, we first need to single out a representative point in $\mathcal{X}_{P,Q}(k)$ where $P$ is a lattice point in $[0,k]^3$ and $Q$ a lattice $(3,k)$\nobreakdash-triangle whose affine hull does not contain $P$. This can be done as follows.

\begin{lem}\label{DLP3.sec.3.lem.2}
For any lattice point $P$ in $[0,k]^3$ and any lattice $(3,k)$-triangle $Q$ such that $P$ and $Q$ are disjoint, there exists a point $x$ in $\mathcal{X}_{P,Q}(k)$ such that
\begin{enumerate}
\item[(i)] $x_1$, $x_2$, and $x_4$ to $x_6$ are non-negative,
\item[(ii)] $x_1x_5$ is at least $x_2x_4$, and
\item[(iii)] $x_4$ is at most $x_2$.
\end{enumerate}
\end{lem}
\begin{proof}
Consider a lattice point $P$ contained in $[0,k]^3$ and a lattice $(3,k)$-triangle $Q$ such that $P$ and $Q$ are disjoint. Let us identify $\mathbb{R}^2$ with the plane spanned by the first two coordinates of $\mathbb{R}^3$ and denote by $\pi:\mathbb{R}^3\rightarrow\mathbb{R}^2$ the orthogonal projection on $\mathbb{R}^2$. Note that $\pi(Q)$ is a triangle or a line segment.

Consider the smallest rectangle of the form $[a,b]\mathord{\times}[c,d]$ that contains $\pi(Q)$. Since $\pi(Q)$ has at most three vertices, one of the edges of the rectangle $[a,b]\mathord{\times}[c,d]$ does not contain a vertex of $\pi(Q)$ in its relative interior. That edge must then admit one of the vertices of $\pi(Q)$ as an extremity because otherwise $[a,b]\mathord{\times}[c,d]$ could be made smaller while still containing $\pi(Q)$. Thanks to the symmetries of the cube $[0,k]^3$, one can assume that $(a,c)$ is a vertex of $\pi(Q)$ while preserving the distance between $P$ and $Q$. In that case, there is at least one vertex of $Q$ whose orthogonal projection on $\mathbb{R}^2$ is $(a,c)$ and we shall denote one such vertex of $Q$ by $q^0$. Let $q^1$ and $q^2$ be the two vertices of $Q$ distinct from $q^0$. As $\pi(q^0)$ is equal to $(a,c)$ and both $\pi(q^1)$ and $\pi(q^2)$ are contained in $[a,b]\mathord{\times}[c,d]$, the first two coordinates of both $q^1-q^0$ and $q^2-q^0$ are non-negative.

Now denote by $A$ the $3\mathord{\times}2$ matrix whose first column is $q^1-q^0$ and whose second column is $q^2-q^0$. Further denote by $b$ the vector $q^0-P$ and by $x$ the lattice point in $\mathcal{X}_{P,Q}(k)$ obtained from $A$ and $b$ via (\ref{DLP3.sec.3.eq.4}) and (\ref{DLP3.sec.3.eq.5}). Since the first two coordinates of $q^1-q^0$ and $q^2-q^0$ are non-negative, $x_1$, $x_2$, $x_4$, and $x_5$ are non-negative. Now observe that $x_1x_5$ may be less than $x_2x_4$. However, one can recover this property by exchanging the labels of $q^1$ and $q^2$ which results in a point $x$ that still belongs to $\mathcal{X}_{P,Q}(k)$ by the definition of this set. Note that exchanging these labels does not affect the sign of $x_1$, $x_2$, $x_4$, or $x_5$.

Similarly, one can make $x_4$ at most $x_2$ by exchanging $q^1$ with $q^2$ and by permuting the first two coordinates of $\mathbb{R}^3$. Note that the combination of these two operations do not affect the signs of $x_1$, $x_2$, $x_4$, $x_5$, or $x_1x_5-x_2x_4$ and results in a point $x$ that still belongs to $\mathcal{X}_{P,Q}(k)$.

Finally, if the third coordinate of $q^2-q^0$ (and therefore $x_6$) is negative, we can make it non-negative without changing the sign of $x_1$, $x_2$, $x_4$, $x_5$, $x_1x_5-x_2x_4$, or $x_2-x_4$ by replacing both $P$ and $Q$ with their symmetric with respect to
$$
\{x\in\mathbb{R}^3:x_3=k/2\}\mbox{.}
$$

This results in a point $x$ in $\mathcal{X}_{P,Q}(k)$ with the desired properties.
\end{proof}

Assuming that the distance of $P$ and $Q$ is equal to $\varepsilon_0(3,k)$, we will bound the coordinates of the points $x$ in $\mathcal{X}_{P,Q}(k)$ that satisfy the assertions (i), (ii), and (iii) in the statement of Lemma~\ref{DLP3.sec.3.lem.2}. In order to do that, we will use the following three lower bounds on the polynomial expression
$$
r(k)=k^4-4k^3+4k^2-2k+1
$$
whose straightforward proofs, that consist in solving polynomial inequalities of degree at most $4$ in $k$ or in $\sqrt{2k+1}$, are omitted.

\begin{prop}\label{DLP3.sec.3.prop.-10}
If $k$ is at least $9$, then
\begin{enumerate}
\item[(i)] $r(k)$ is greater than $k^2(k-3)^2$,
\item[(ii)] $r(k)$ is at least $\bigl(k^2-2k-1\bigr)^2$, and
\item[(iii)] $r(k)$ is at least $\bigl(k^2-4(k-\sqrt{2k+1})\bigr)^2$,
\end{enumerate}
\end{prop}

The announced bounds on the coordinates of $x$ are given depending on the sign of $x_3$. The following lemma treats the case when $x_3$ is non-negative.

\begin{lem}\label{DLP3.sec.3.lem.3}
Assume that $k$ is at least $9$. Consider a lattice point $P$ contained in $[0,k]^3$ and a lattice $(3,k)$-triangle $Q$ such that the distance of $P$ to the affine hull of $Q$ is $\varepsilon_0^u(3,k)$. Further consider a point $x$ in $\mathcal{X}_{P,Q}(k)$ such that
\begin{enumerate}
\item[(i)] $x_1$, $x_2$, and $x_4$ to $x_6$ are non-negative,
\item[(ii)] $x_1x_5$ is at least $x_2x_4$, and
\item[(iii)] $x_4$ is at most $x_2$.
\end{enumerate}

If $x_3$ is non-negative, then $x_1$, $x_5$, and $x_6$ are at least $k-2$ while $x_4$ is at most $3$. Moreover, either $x_2$ is at least $k-2$ and $x_3$ is at most $3$ or inversely, $x_3$ is at least $k-2$ and $x_2$ is at most $3$.
\end{lem}
\begin{proof}
Under the assumption that the distance between $P$ and the affine hull of $Q$ is equal to $\varepsilon_0^u(3,k)$, we obtain from Proposition~\ref{DLP3.sec.3.prop.0} that
$$
\varepsilon_0^u(3,k)=\frac{1}{\sqrt{g(x)}}
$$
and, in turn, by (\ref{DLP3.sec.3.eq.2}) that
\begin{equation}\label{DLP3.sec.3.lem.3.eq.0}
g(x)\geq3k^4-4k^3+4k^2-2k+1\mbox{.}
\end{equation}

It then follows from (\ref{DLP3.sec.3.eq.7}) and Proposition \ref{DLP3.sec.3.prop.-0.5} that
\begin{equation}\label{DLP3.sec.3.lem.3.eq.1}
(x_1x_5-x_2x_4)^2\geq{r(k)}\mbox{.}
\end{equation}

Since all the $x_i$ in the left-hand side of (\ref{DLP3.sec.3.lem.3.eq.1}) are non-negative and $x_1x_5$ is at least $x_2x_4$, it follows that $x_1^2x_5^2$ is at least $r(k)$. Now recall that the absolute value of $x_1$ and $x_5$ is at most $k$. As, according to Proposition \ref{DLP3.sec.3.prop.-10}, $r(k)$ is greater than $k^2(k-3)^2$, it follows that $x_1$ and $x_5$ are both at least $k-2$.

Likewise, it follows from (\ref{DLP3.sec.3.lem.3.eq.1}) that
$$
\bigl(k^2-x_2x_4\bigr)^2\geq{r(k)}\mbox{.}
$$

However, according to Proposition \ref{DLP3.sec.3.prop.-10}, $r(k)$ is at least $(k^2-2k-1)^2$ and it follows that $x_2x_4$ is at most $2k+1$. In summary, we have established that
\begin{equation}\label{DLP3.sec.3.lem.3.eq.3}
\left\{
\begin{array}{l}
x_1\geq{k-2}\mbox{,}\\
x_5\geq{k-2}\mbox{,}\\
x_2x_4\leq2k+1\mbox{.}\\
\end{array}
\right.
\end{equation}

Since $x_4$ is at most $x_2$, it follows from the third inequality that $x_4$ is at most $\sqrt{2k+1}$. Now, according to (\ref{DLP3.sec.3.lem.3.eq.0}), (\ref{DLP3.sec.3.eq.7}), and Proposition \ref{DLP3.sec.3.prop.-0.5},
\begin{equation}\label{DLP3.sec.3.lem.3.eq.2}
(x_1x_6-x_3x_4)^2\geq{r(k)}\mbox{.}
\end{equation}

Observe that $x_1x_6$ must be at least $x_3x_4$. Indeed, otherwise, as $x_3$ is at most $\sqrt{2k+1}$, the left-hand side of (\ref{DLP3.sec.3.lem.3.eq.2}) would be less than $k^2(2k+1)$ and therefore less than $k^2(k-3)^2$, which would contradict Proposition \ref{DLP3.sec.3.prop.-10}. As $x_1x_6$ is at least $x_3x_4$, borrowing the argument we used to prove (\ref{DLP3.sec.3.lem.3.eq.3}) yields
\begin{equation}\label{DLP3.sec.3.lem.3.eq.4}
\left\{
\begin{array}{l}
x_6\geq{k-2}\mbox{,}\\
x_3x_4\leq2k+1\mbox{.}\\
\end{array}
\right.
\end{equation}

In the rest of the proof we consider two cases depending on whether $x_2x_6$ is at least $x_3x_5$ or not. First, if $x_2x_6$ is at least $x_3x_5$, the argument that we used to establish both (\ref{DLP3.sec.3.lem.3.eq.3}) and (\ref{DLP3.sec.3.lem.3.eq.4}) further yields
\begin{equation}\label{DLP3.sec.3.lem.3.eq.5}
\left\{
\begin{array}{l}
x_2\geq{k-2}\mbox{,}\\
x_3x_5\leq2k+1\mbox{.}\\
\end{array}
\right.
\end{equation}

As $x_2$ is at least $k-2$, it follows from (\ref{DLP3.sec.3.lem.3.eq.3}) that
\begin{equation}\label{DLP3.sec.3.lem.3.eq.6}
x_4\leq\frac{2k+1}{k-2}
\end{equation}
and as $x_5$ is at least $k-2$, it follows from (\ref{DLP3.sec.3.lem.3.eq.5}) that
\begin{equation}\label{DLP3.sec.3.lem.3.eq.7}
x_3\leq\frac{2k+1}{k-2}\mbox{.}
\end{equation}

Since $k$ is at least $9$, the right-hand side of (\ref{DLP3.sec.3.lem.3.eq.6}) and (\ref{DLP3.sec.3.lem.3.eq.7}) is at most $3$ and we obtain that $x_1$, $x_2$, $x_5$, and $x_6$ are all at least $k-2$ while $x_3$ and $x_4$ are at most $3$, as desired. Finally if $x_2x_6$ is less than $x_3x_5$, then using the same argument but where $x_2$ is exchanged with $x_3$ and $x_5$ with $x_6$ shows that $x_1$, $x_3$, $x_5$, and $x_6$ are all at least $k-2$ while $x_2$ and $x_4$ are at most $3$.
\end{proof}

In order to prove a statement similar to that of Lemma \ref{DLP3.sec.3.lem.3} but in the case when $x_3$ is negative, we will use of the following constraint satisfied by the points in $\mathcal{X}_{P,Q}(k)$ that arises from the embedding of $P$ and $Q$ into $[0,k]^3$.

\begin{prop}\label{DLP3.sec.3.prop.1}
Consider a lattice point $P$ in $[0,k]^3$ and a lattice $(3,k)$\nobreakdash-triangle $Q$ such that $P$ and $Q$ are disjoint. For every point $x$ in $\mathcal{X}_{P,Q}(k)$,
$$
|x_3-x_6|\leq{k}\mbox{.}
$$
\end{prop}
\begin{proof}
Consider a point $x$ in $\mathcal{X}_{P,Q}(k)$. By the definition of $\mathcal{X}_{P,Q}(k)$, there exist a lattice point $\tilde{P}$ in $[0,k]^3$ and a lattice $(3,k)$-triangle $\tilde{Q}$ that coincide with $P$ and $Q$ up to some isometry of $[0,k]^3$ and satisfy
$$
\left\{
\begin{array}{l}
x_3=\tilde{q}^1_3-\tilde{q}^0_3\mbox{ and}\\
x_6=\tilde{q}^2_3-\tilde{q}^0_3\mbox{,}\\
\end{array}
\right.
$$
where $\tilde{q}^0$, $\tilde{q}^1$, and $\tilde{q}^2$ adequately label the vertices of $\tilde{Q}$. Therefore
$$
x_3-x_6=\tilde{q}^1_3-\tilde{q}^2_3
$$
and since the points $\tilde{q}^0$, $\tilde{q}^1$, and $\tilde{q}^2$ are contained in the cube $[0,k]^3$, the absolute value of $x_3-x_6$ is at most $k$, as desired.
\end{proof}

Thanks to the constraint stated by Proposition \ref{DLP3.sec.3.prop.1}, we can bound the coordinates of the points $x$ in $\mathcal{X}_{P,Q}(k)$ that satisfy the assertions (i), (ii), and (iii) in the statement of Lemma~\ref{DLP3.sec.3.lem.2} and whose third coordinate is negative.

\begin{lem}\label{DLP3.sec.3.lem.4}
Assume that $k$ is at least $9$. Consider a lattice point $P$ contained in $[0,k]^3$ and a lattice $(3,k)$-triangle $Q$ such that the distance of $P$ to the affine hull of $Q$ is $\varepsilon_0(3,k)$. Further consider a point $x$ in $\mathcal{X}_{P,Q}(k)$ such that
\begin{enumerate}
\item[(i)] $x_1$, $x_2$, and $x_4$ to $x_6$ are non-negative,
\item[(ii)] $x_1x_5$ is at least $x_2x_4$, and
\item[(iii)] $x_4$ is at most $x_2$.
\end{enumerate}

If $x_3$ is negative, then $x_1$ and $x_5$ are at least $k-2$ while $x_2$ and $x_6$ are at least $k-6$. Moreover $x_3$ is at least $-4$ and $x_4$ is at most $6$.
\end{lem}
\begin{proof}
The proof begins exactly like that of Lemma \ref{DLP3.sec.3.lem.3}. In particular, we get
\begin{equation}\label{DLP3.sec.3.lem.4.eq.0}
g(x)\geq3k^4-4k^3+4k^2-2k+1
\end{equation}
and
$$
\left\{
\begin{array}{l}
x_1\geq{k-2}\mbox{,}\\
x_5\geq{k-2}\mbox{,}\\
x_2x_4\leq2k+1\mbox{.}\\
\end{array}
\right.
$$

As $x_4$ is at most $x_2$, it follows that $x_4$ is at most $\sqrt{2k+1}$. Now, combining (\ref{DLP3.sec.3.eq.7}), (\ref{DLP3.sec.3.lem.4.eq.0}), and Proposition \ref{DLP3.sec.3.prop.-0.5}, we obtain
\begin{equation}\label{DLP3.sec.3.lem.4.eq.3.5}
(x_1x_6-x_3x_4)^2\geq{r(k)}
\end{equation}

Under the assumption that $x_3$ is negative, since $x_1$ is at most $k$, $x_4$ at most $\sqrt{2k+1}$, and, by Proposition \ref{DLP3.sec.3.prop.1}, $x_6$ at most $k+x_3$, this implies that
$$
\bigl(k^2+\bigl(k-\sqrt{2k+1}\bigr)x_3\bigr)^2\geq{r(k)}\mbox{.}
$$

Developing the left-hand side and expressing $r(k)$ in terms of $k$ yields
$$
\bigl(k-\sqrt{2k+1}\bigr)^2x_3^2+2k^2\bigl(k-\sqrt{2k+1}\bigr)x_3+4k^3-4k^2+2k-1\geq0\mbox{.}
$$

Treating this as a quadratic inequality in $(k-\sqrt{2k+1})x_3$, the discriminant is $4r(k)$ which is positive because $k$ is at least $9$. Hence, either
\begin{equation}\label{DLP3.sec.3.lem.4.eq.4}
\bigl(k-\sqrt{2k+1}\bigr)x_3\leq-k^2-\sqrt{r(k)}
\end{equation}
or
\begin{equation}\label{DLP3.sec.3.lem.4.eq.5}
\bigl(k-\sqrt{2k+1}\bigr)x_3\geq-k^2+\sqrt{r(k)}\mbox{.}
\end{equation}

However, by Proposition \ref{DLP3.sec.3.prop.-10}, $r(k)$ is at least $k^2(k-3)^2$ and (\ref{DLP3.sec.3.lem.4.eq.4}) implies
$$
x_3\leq-k\frac{2k-3}{k-\sqrt{2k+1}}\mbox{.}
$$

As the right-hand side of this inequality is less than $-k$ and $x_3$ is at least $-k$, this inequality cannot hold and this proves that (\ref{DLP3.sec.3.lem.4.eq.5}) does. Hence,
\begin{equation}\label{DLP3.sec.3.lem.4.eq.6}
x_3\geq\frac{\sqrt{r(k)}-k^2}{k-\sqrt{2k+1}}\mbox{.}
\end{equation}

By Proposition \ref{DLP3.sec.3.prop.-10}, the numerator of the right-hand side is greater than minus four times its denominator and we get that $x_3$ is at least $-4$.

Let us now bound $x_6$. As $x_1$ and $x_4$ are both at most $k$, (\ref{DLP3.sec.3.lem.4.eq.3.5}) implies that $r(k)$ is at most $k^2(x_6-x_3)^2$ and therefore at most $k^2(x_6+4)^2$. This proves that $x_6$ is at least $k-6$ as by Proposition \ref{DLP3.sec.3.prop.-10}, $r(k)$ is greater than $k^2(k-3)^2$.

We combine (\ref{DLP3.sec.3.eq.7}), (\ref{DLP3.sec.3.lem.4.eq.0}), and Proposition \ref{DLP3.sec.3.prop.-0.5} again in order to obtain
$$
(x_2x_6-x_3x_5)^2\geq{r(k)}
$$
which, as $x_5$ and $x_6$ are at most $k$ and $x_3$ is at least $-4$ implies
$$
x_2\geq\frac{\sqrt{r(k)}}{k}-4
$$

By Proposition \ref{DLP3.sec.3.prop.-10}, $r(k)$ is greater than $k^2(k-3)^2$ and therefore $x_2$ is at least $k-6$. Finally, as by (\ref{DLP3.sec.3.lem.3.eq.3}), $x_2x_4$ is at most $2k+1$,
$$
x_4\leq\frac{2k+1}{k-6}\mbox{.}
$$

Since $k$ is at least $9$, the right-hand side of this inequality is less than $7$. Therefore $x_4$ is not greater than $6$, as desired.
\end{proof}

Combining Lemmas \ref{DLP3.sec.3.lem.2}, \ref{DLP3.sec.3.lem.3} and \ref{DLP3.sec.3.lem.4}, we now prove Theorem \ref{DLP3.sec.3.thm.1}.

\begin{proof}[Proof of Theorem \ref{DLP3.sec.3.thm.1}]
Assume that $k$ is at least $9$ and that the distance between $P$ and the affine hull of $Q$ is $\varepsilon_0^u(3,k)$. By Lemma \ref{DLP3.sec.3.lem.2}, there exists a point $x$ in $\mathcal{X}_{P,Q}(k)$ whose first six coordinates are non-negative, except maybe the third, while $x_1x_5$ is at least $x_2x_4$ and $x_2$ is at least $x_4$. We consider two cases depending on the sign of $x_3$ corresponding to either Lemma \ref{DLP3.sec.3.lem.3} or Lemma \ref{DLP3.sec.3.lem.4}.

First, if $x_3$ is non-negative then, it follows from Lemma \ref{DLP3.sec.3.lem.3} that $x_4$ is at most $3$ while $x_1$, $x_5$, and $x_6$ are at least $k-2$. Moreover, either $x_2$ is at least $k-2$ and $x_3$ is at most $3$ or $x_3$ is at least $k-2$ and $x_2$ is at most $3$. Observe that exchanging $x_i$ and $x_{i+1}$ when $i$ is equal to $2$,$5$, and $8$ results in a lattice point whose coordinates satisfy the same bounds. Moreover, the resulting point still belongs to $\mathcal{X}_{P,Q}(k)$ because these three exchanges amount to transform $P$ and $Q$ by permuting the second and third coordinates of $\mathbb{R}^3$. Hence, we can assume without loss of generality that $x$ in a point from $\mathcal{X}_{P,Q}(k)$ such that $x_1$, $x_2$, $x_5$, and $x_6$ are at least $k-2$ while $x_3$ and $x_4$ are both at most $3$.

If however $x_3$ is negative. It then follows from Lemma \ref{DLP3.sec.3.lem.4} that $x_1$, $x_2$, $x_5$, and $x_6$ are at least $k-6$ while $x_3$ is at least $-4$ and $x_4$ is at most $6$.

In particular, in both cases, the lattice point $a$ from $\mathbb{R}^6$ such that $a_i$ is equal to $k-x_i$ when $i$ is equal to $1$, $2$, $5$, or $6$ and to $x_i$ when $i$ is equal to $3$ or $4$ belongs to $[0,6]^2\mathord{\times}[-4,6]\mathord{\times}[0,6]^3$ and by construction, $\phi_k(a)$ is precisely the orthogonal projection of $x$ on $\mathbb{R}^6$. In order to prove that $a$ belongs to $\mathcal{A}$, it therefore suffices to show that $a_6$ is at least $-a_3$. According to Proposition \ref{DLP3.sec.3.prop.1}, $x_6-x_3$ is at most $k$ which in terms of the coordinates of $a$ implies that $a_6$ is at least $-a_3$. As a consequence, $a$ belongs to $\mathcal{A}$, as desired.
\end{proof}

\bibliography{Bessatsu}
\bibliographystyle{ijmart}

\end{document}